\documentclass[12pt]{amsart}
\usepackage{amssymb,graphicx}

\newcommand{\LL}{\mathcal L}

\newcommand{\Z}{\mathbb Z}
\newcommand{\N}{\mathbb N}

\DeclareMathOperator{\diam}{diam}

\newtheorem{thm}{Theorem}
\newtheorem{lem}[thm]{Lemma}
\newtheorem{cor}[thm]{Corollary}

\begin{document}
\title[MET and relative equilibrium states]{A multiplicative ergodic theoretic
characterization of relative equilibrium states}
\author{John Antonioli, Soonjo Hong and Anthony Quas}
\thanks{For SH, funding was provided by Hongik University, Korea}
\thanks{AQ's research was partially supported by NSERC}
\maketitle
\begin{abstract}
In this article, we continue the structural study of factor maps between symbolic dynamical 
systems and the relative thermodynamic formalism. Here, one is studying a factor map
from a shift of finite type $X$ (equipped with a potential function) to a sofic shift $Z$, 
equipped with a shift-invariant measure $\nu$. We study relative equilibrium states, that is
shift-invariant measures on $X$ that push forward under the factor map to $\nu$ which 
maximize the relative pressure: the relative entropy plus the integral of $\phi$. 


In this paper, we establish a new connection to multiplicative ergodic theory by
relating these factor triples to a cocycle of Ruelle Perron-Frobenius operators,
and showing that the principal Lyapunov exponent of this cocycle is the relative pressure; and the 
dimension of the leading Oseledets space is equal to the number of measures of relative
maximal entropy, counted with a previously-identified concept of multiplicity. 
\end{abstract}

\section{Introduction}

Let $A$ and $B$ be finite non-empty sets, let $X\subset A^\Z$ be an irreducible shift of finite type, 
and let $\pi:X\to B^\Z$ be a shift-commuting
map so that $Z=\pi(X)$ is a sofic factor of $X$. 
Given a shift-invariant
measure $\nu$ on $Z$, we are interested in $\pi_*^{-1}\{\nu\}$, the set of shift-invariant measures on $X$
whose push-forward is $\nu$. \emph{Relative thermodynamic formalism} gives a means to identify distinguished 
elements of $\pi_*^{-1}\{\nu\}$ similar to standard thermodynamic formalism. In fact, standard thermodynamic
formalism is the special case of relative thermodynamic formalism where $Z$ is the one-point space. 

We make a standing assumption that the 
factor map has the property
that $\pi(x)_0$ only depends on $x_0,x_1,\ldots$ (and in fact in this case, 
by the Curtis-Hedlund-Lyndon theorem, 
$\pi(x)_0$ only depends on $x_0,\ldots,x_{k-1}$ for some $k\in\N_0$).
We call this a \emph{forward-looking factor map}.
The simplest case of such factor maps is the case where $\pi(x)_0$ depends only on 
$x_0$, that is $\pi$ is a one-block factor map. It is well-known (\cite[Proposition 1.5.12]{LindMarcus}) that up to 
conjugacy, this is the general case. 
Let $X^+\subset A^{\N_0}$ and $Z^+\subset B^{\N_0}$ denote the one-sided versions of $X$ and $Z$.
By the assumption on the factor map, $\pi$ induces a map from $X^+$ to $Z^+$, which we also call $\pi$. 
 
Given an invariant measure $\nu$ on $Z$,
and a H\"older continuous function $\phi$ on $X^+$ (which we call a \emph{potential}
and which we also view as a function on $X$),
recall that a \emph{relative
equilibrium state} of $\phi$ over $\nu$ is an invariant measure $\mu$ on $X$ such that 
$\pi_*(\mu)=\nu$ and 
$h_\mu+\int \phi\,d\mu=r_\pi(\nu)$, where
$r_\pi(\nu):=\max_{\{\lambda\colon \pi_*(\lambda)=\nu\}}\big(h_\lambda+\int
\phi\,d\lambda\big)$. Since $\phi$ is continuous, existence of a relative equilibrium state
follows from compactness of the space of measures and upper semi-continuity of entropy. 
A number of papers \cite{PQS, MahsaAQ, AAY} have given bounds on the 
number of ergodic relative equilibrium states and measures of relative maximal entropy (relative equilibrium
states in the case where $\phi=0$)
in this setting.

We are seeking
to relate the number of relative equilibrium states of $\phi$ over $\nu$
to terms of the Lyapunov exponents and Oseledets spaces of a 
cocycle of Perron-Frobenius operators that we describe below. 
For $0<\beta<1$, we
introduce a metric $d_\beta$ on $X^+$ given by 
$d_\beta(x,x')=\beta^{\min\{n\colon x_n\ne x'_n\}}$ and write $C^\beta(X^+)$ for the
Lipschitz functions with respect to this metric. For $\beta\in [\frac 12,1)$, $C^\beta(X^+)$
is precisely the collection of $(-\log_2\beta)$-H\"older continuous functions with respect to
the standard metric $d_{1/2}$. More generally, as is common in symbolic dynamics,
we refer
to the spaces $C^\beta(X^+)$ as the H\"older continuous functions on $X^+$ (even when $0<\beta<\frac 12$).
Once the potential $\phi$ is fixed, we
choose $\beta$ so that $\phi\in C^\beta(X^+)$.

For $j\in B$, define an operator $\LL_j$ on $C^\beta(X^+)$ by
$$
\LL_jf(x)=\sum_{\{i\colon ix\in X^+;(\pi(ix))_0=j\}}
e^{\phi(ix)}f(ix),
$$
where $ix$ denotes the point in $A^{\N_0}$ defined by $(ix)_0=i$, 
$(ix)_n=x_{n-1}$ for $n\ge 1$. Similarly for a word $w=w_{0}\ldots w_{k-1}$, 
$(wx)_n=w_{n}$ if $n<k$ and $(wx)_n=x_{n-k}$ if $n\ge k$.
If $x$ is an element of $X$ or $X^+$, we use the notation $x_0^{n-1}$ to denote the word $x_0\ldots x_{n-1}$. 
We shall study the cocycle over the dynamical system $\sigma\colon Z\to Z$
where the map corresponding to $z$ is $\LL_z\colon=\LL_{z_0}$. As usual, we define
$\LL_z^{(n)}=\LL_{\sigma^{n-1}z}\circ\ldots\circ \LL_z$.

An inductive calculation shows that
$$
\LL_z^{(n)}f(x)=\sum_{\{w\colon (\pi(wx))_0^{n-1}=z_0^{n-1}, w_{0}\ldots w_{k-1}x\in X^+\}}
e^{S_n\phi(wx)}f(wx),
$$
where as usual, $S_n\phi(wx)$ denotes the sum $\phi(wx)+\ldots+\phi(T^{n-1}(wx))$. 

Our main theorem states that for an ergodic invariant measure $\nu$ on $Z$ and a H\"older continuous
potential, the number of ergodic relative equilibrium states on $X$ 
is the multiplicity of the top Lyapunov exponent
of the above cocycle. While we defer detailed definitions and statements which the theorem relies on,
we mention Theorems \ref{thm:Yoo-factor} and \ref{thm:Yoo-mult} due to Jisang Yoo which establish that a 
factor map $\pi:X\to Z$ of the type that we consider 
may be expressed as a composition of factor maps $\pi_1\colon X\to Y$ and
$\pi_2\colon Y\to Z$ with nice properties defined in detail below: $\pi_1$ is of class degree 1,
and $\pi_2$ is finite-to-one of degree $c_\pi$ which is the class degree of the original map $\pi$. 
Recall that for a finite-to-one factor map $\pi_2\colon Y\to Z$ 
from one irreducible two-sided sofic system to another, the \emph{degree} of $\pi_2$
is the minimal cardinality of $\pi_2^{-1}(z)$ as $z$ runs over $Z$. The minimum is attained for
all doubly transitive (i.e. right and left transitive) points (see \cite[Theorem 9.1.11]{LindMarcus}). 

If $\nu$ is a fully-supported ergodic shift-invariant measure on $Z$, then since $\pi_2\colon Y\to Z$
is of degree $c_\pi$, $\nu$-almost every $\omega\in Z$ has $c_\pi$ pre-images. It may be shown that
there are only finitely many ergodic invariant measures on $Y$ that factor onto $\nu$,
$\nu_1,\ldots,\nu_k$ say.
Yoo defines \emph{multiplicities} $m_1,\ldots,m_k$ of these measures with $m_1+\ldots+m_k=c_\pi$
and shows that for $\nu$-almost every $z\in Z$, of the $c_\pi$ elements of $\pi_2^{-1}(z)$,
$m_i$ are generic for $\nu_i$ for each $i$.

Since the multiplicative ergodic theory of infinite-dimensional operators is less well-known than in the finite-dimensional 
case, we include a quick summary. While there are multiple versions of operator-valued multiplicative ergodic theorems,
we focus on the context in this article. Assume that there is a `base' dynamical system $\sigma\colon Z\to Z$
which is a continuous homeomorphism from a compact metric space to itself. The space $Z$
is assumed to be equipped with a $\sigma$-invariant ergodic Borel probability measure $\nu$.  
A function on $Z$ is said to be \emph{$\nu$-continuous} if for any $\epsilon>0$, there exists a subset
$Z'\subset Z$ of measure at least $1-\epsilon$ on which the restriction of the function is continuous. 
Recall that by Lusin's theorem, any Borel-measurable real-valued function on $Z$ is $\nu$-continuous, 
but this is not necessarily true for functions with non-separable ranges. 
There is also a Banach space $B$ and a collection
$\{\mathcal L_z\colon z\in Z\}$ of linear maps from $B$ to itself. One then studies the operators $\mathcal L_z^{(n)}$, 
defined by $\LL_z^{(n)}=\LL_{\sigma^{n-1}z}\circ\ldots\circ\LL_z$. Under conditions of quasi-compactness (which are satisfied 
in our context), there exists a leading Lyapunov exponent $\lambda_1$, an exponent $\lambda_2<\lambda_1$,
a \emph{multiplicity} $M$, a $\nu$-continuous map $E$
from $Z$ into $\mathcal G_{M}(B)$, the collection of $M$-dimensional subspaces of $B$, and $\nu$-continuous maps $\eta_1,\ldots,\eta_{M}$ from $Z$ into $B^*$
satisfying
\begin{itemize}
\item(equivariance) $\LL_z(E(z))=E(\sigma(z))$, $\nu$-a.e.; and $\LL_z(F(z))\subset F(\sigma(z))$, where $F(z)=\bigcap_{i=1}^M\ker\eta_i(z)$, 
$\nu$-a.e.;
\item(growth) $\lim_{n\to\infty}\frac 1n\log\|\LL_z^{(n)}f\|=\lambda_1$ for all $f\in E(z)\setminus\{0\}$; and
$\lim_{n\to\infty}\frac 1n\log\|\LL_z^{(n)}|_{F(z)}\|=\lambda_2$.
\end{itemize}

A one sentence summary of this is that over a.e.\ $z$, $B$ decomposes into two equivariant spaces of dimension $M$ and co-dimension $M$ 
respectively on which the growth rates of the operator cocycle are $\lambda_1$ and $\lambda_2$ respectively. 

Some of the terms appearing in the statement below of the main theorem will be defined in Section \ref{sec:background}.
\begin{thm}\label{thm:main}
Let $\pi$ be a forward-looking factor map from an irreducible (two-sided) shift of finite type $X$ to a 
sofic shift $Z$ and let $\phi\in C^\beta(X^+)$. 
Let $\pi=\pi_2\circ\pi_1$ be the factorization described above where 
$\pi_1\colon X\to Y$ is of class degree one and
map $\pi_2:Y\to Z$ has degree equal to the class degree of $\pi$.

Let $\nu$ be a fully supported ergodic invariant measure on $Z$ and let $\nu_1$,\ldots,$\nu_k$ be 
the lifts of $\nu$ under $\pi_2$ with multiplicities $m_1,\ldots,m_k$ respectively.
Let $r_{\pi_1}(\nu_i)$ 
be the relative pressure of $\phi$ over $\nu_i$ and $r_\pi(\nu)$ be the relative pressure of $\phi$ over $\nu$.

Let $(\LL_z^{(n)})$ be the cocycle of linear operators over $Z$ acting on $C^\beta(X^+)$ described above. 
Then the largest Lyapunov exponent of the cocycle is $r_\pi(\nu)$ and the multiplicity of this
exponent is
$$
\sum_{r_{\pi_1}(\nu_i)=r_\pi(\nu)}m_i.
$$
\end{thm}

In the case where $\phi$ is locally constant, we can give a more tangible description of this multiplicity
as the multiplicity of the leading exponent of an explicit matrix cocycle.

\begin{cor}\label{cor:locconst}
Let $\pi$, $X$, $Z$ and $\nu$ be as in the statement of Theorem \ref{thm:main} and suppose
additionally that $\phi$ is locally constant. Then the action of $\LL_z^{(n)}$ restricts
to the space of locally constant functions (constant on cylinders of the same length).
The multiplicity of the leading 
exponent of this cocycle is the same as that of the cocycle in Theorem \ref{thm:main}.
\end{cor}

In the proof of this corollary, we assume (without loss of generality) that $X$ is a 1-step shift
of finite type and $\phi(x)$ depends only on 
$x_0$. In this case, the locally constant functions in the proof also depend only on the 0th coordinate. 
It is then straightforward to write down the matrix cocycle representing the action of
$\LL_z^{(n)}$ on these functions.

\section{An Example}
In this section, we give an example illustrating the objects appearing in the theorem
and the corollary. The example is mostly self-contained, but we rely on one fact from the proof of the main
theorem while presenting the example.
In fact, where the corollary would give a cocycle of $3\times 3$ matrices,
we are able to exploit some symmetries to build instead a cocycle of $2\times 2$ matrices.
Let $X=\{0,1,2\}^\Z$ and $Z=\{F,N\}^\Z$. The factor map is defined by 
$\pi(x)_j=F$ (for flip) if $x_j$ and $x_{j+1}$ have opposite parities and
$\pi(x)_j=N$ if $x_j$ and $x_{j+1}$ have the same parity (modulo 2).
We define the potential $\phi$ by $\phi(x)=x_0$. 
For any $z\in Z$, the pre-image set $\pi^{-1}\{z\}$ consists of two classes, one consisting of points where
$x_0$ is 0 or 2 and the other consisting of points with $x_0=1$. These two classes are \emph{mutually separated}:
at each $j\in\Z$, points in one class have even symbols while points in the other class have odd symbols, 
so that $\pi$ has class degree 2. A suitable factorization of $\pi$ into the composition of a map of class degree
1 and a map of degree 2 is given by $\pi=\pi_2\circ\pi_1$ where $Y=\{0,1\}^\Z$,
$\pi_1(x)_i=x_i\mod 2$ and $\pi_2$ is $\pi|_Y$. To see that 
$\pi_1$ is of class degree 1, notice that if $x,x'\in \pi_1^{-1}(y)$, then the hybrid point $\bar x$ agreeing 
with $x$ on symbols up to the $(k-1)$st and agreeing with $x'$ thereafter belongs to $X$ and is a 
pre-image of $y$, so that $x$ transitions to $x'$ for any two elements of $\pi_1^{-1}(y)$ (as defined in Section
\ref{sec:background}). To see that
$\pi_2$ has degree 2, notice that each $z\in Z$ has exactly two pre-images, one the image of the
other under the map $r\colon Y\to Y$ given by $r(y)_j=1-y_j$. 

If $\mu_p$ is the Bernoulli measure on $Y$ with 0's with probability $p$ and 1's with 
probability $1-p$, then $\mu_p\circ\pi_2^{-1}=\mu_{1-p}\circ\pi_2^{-1}$ (this follows from the 
facts that $\pi_2\circ r=\pi_2$ and
$\mu_p\circ r^{-1}=\mu_{1-p}$). We let $\nu_p=\mu_p\circ \pi_2^{-1}$. 
It is not hard to see that if $p=\frac 12$, then $\nu_p$ is the uniformly distributed Bernoulli
measure on $Z$. On the other hand, for $p\ne \frac 12$, the measure $\nu_p$ is 
a Hidden Markov Measure where there is long range dependence between the 
entries (see for example \cite{Blackwell}). 

We then look at the equilibrium states on $X$ for $\phi$ relative to the factor $\nu_p$ on $Z$. 
We find these by first understanding the lifts of $\nu_p$ to $Y$ under $\pi_2$.
The ergodic lifts of $\nu_p$ to $Y$ are $\mu_p$ and $\mu_{1-p}$, each with multiplicity 1
in the case $p\ne\frac 12$;
and $\mu_{\frac 12}$ with multiplicity 2 in the case $p=\frac 12$. To understand this, notice that a typical 
point of $\nu_{\frac 12}$ has two preimages, both generic for the Bernoulli measure $\mu_{\frac 12}$ on $Y$
with the image of the other under $r$.

To find the relative equilibrium states for $\phi$ over $(Z,\nu_p)$ with the factor map $\pi$,
we then look for the relative equilibrium states of $\phi$ over $(Y,\mu_p)$ and
$(Y,\mu_{1-p})$ with the factor map $\pi_1$. By \cite{AAY}, since $\pi_1$ has class degree 1,
there is a unique relative equilibrium state of $\phi$ over $(Y,\mu_p)$ for each $p$. 
The relative pressures with the factor map $\pi_1$ over $(Y,\mu_p)$ and $(Y,\mu_{1-p})$
are $p\log(1+e^2)+(1-p)$ and $(1-p)\log(1+e^2)+p$. To see this, notice
that to lift $(Y,\mu_p)$, the symbol $0$ in $Y$ with probability $p$ is to be split into two states $0$ and $2$. 
Choosing $0$ with probability $p/(1+e^{2})$ and $2$ with probability $pe^{2}/(1+e^{2})$ maximizes the
relative pressure. For $p=\frac 12$, the relative equilibrium state is a lift of $\mu_{\frac12}$ to $X$
under $\pi_1$. The relative equilibrium state is the Bernoulli measure with 0, 1 and 2 having weights
$\frac 12/(1+e^2)$, $\frac 12$ and $\frac12 e^2/(1+e^2)$ respectively. The relative pressure
is $\frac 12(1+\log(1+e^2))$.

Since it is not easy to directly compute exponents of Perron-Frobenius operator cocycles, we identify a
finite-dimensional space $V$ of piecewise constant functions, invariant under the cocycle,
and do computations there. 
That this is possible is because the function $\phi$ is piecewise constant. 

We also need to see why the growth rates appearing in the subspace $V$ are the maximal growth
rates in the full Banach space. This follows since $V$ intersects each of the cones $\mathcal C_a\cap S_P$
appearing in Lemma \ref{lem:finitediam} for $P$ satisfying the conditions appearing in Lemma
\ref{lem:positiveset}. 

Let the two-dimensional space $V$ be the collection of
functions on $X$, constant on cylinders of length 1, with the property that the value on the [0] and
[2] cylinders are equal. We
claim that both $\mathcal L_N$ and $\mathcal L_F$ map $V$ into itself. We represent an element
$f$ of $V$ by a vector consisting of its values on the cylinder sets $[0]\cup [2]$ and $[1]$ respectively.
We then compute the action of $\mathcal L_N$ and $\mathcal L_F$ on $V$ as follows: 

We have 
\begin{align*}
\LL_N f(x)&=\begin{cases}
e^{\phi(1x)}f(1x)&\text{if $x_0=1$;}\\
e^{\phi(0x)}f(0x)+e^{\phi(2x)}f(2x)&\text{if $x_0$ is 0 or 2.}
\end{cases}\\
\LL_F f(x)&=\begin{cases}
e^{\phi(0x)}f(0x)+e^{\phi(2x)}f(2x)&\text{if $x_0=1$;}\\
e^{\phi(1x)}f(1x)&\text{if $x_0$ is 0 or 2.}
\end{cases}
\end{align*}

Recalling that $\phi(x)=x_0$ and representing both $f$ and its image by vectors in the order
described above, we have $\LL_N$ and $\LL_F$ are represented on $V$ by the matrices
$$
A_N=\begin{pmatrix}1+e^2&0\\0&e\end{pmatrix}
\quad\text{and}\quad
A_F=\begin{pmatrix}0&e\\1+e^2&0\end{pmatrix}.
$$
Let $A_z^{(n)}$ denote the cocycle over $z$ generated by these matrices. 
If $y$ is a $\mu_p$-generic point, then $r(y)$ is $\mu_{1-p}$-generic and
$z=\pi_2(y)=\pi_2(r(y))$ is $\nu_p$-generic. We can verify that if $y_0\ldots y_n$
starts and ends with even symbols (which implies that $z_0\ldots z_{n-1}$ has an
even number of $F$'s), then 
$$
A_z^{(n)}=\begin{pmatrix}(1+e^2)^Ee^O&0\\0&(1+e^2)^Oe^E\end{pmatrix},
$$
where $E$ denotes the number of even symbols in $y_0\ldots y_n$ and $O$ is the number
of odd symbols.
Similarly if $y_0\ldots y_{n}$ begins and ends with odd symbols
$$
A_z^{(n)}=\begin{pmatrix}(1+e^2)^Oe^E&0\\0&(1+e^2)^Ee^O\end{pmatrix}.
$$
Finally if $y_0\ldots y_n$ begins with an even symbol and ends with an odd symbol, or
begins with an odd symbol and ends with an even symbol then $A_z^{(n)}$ is respectively
$$
\begin{pmatrix}0&(1+e^2)^Ee^O\\(1+e^2)^Oe^E&0\end{pmatrix}
\text{ or }
\begin{pmatrix}0&(1+e^2)^Oe^E\\(1+e^2)^Ee^O&0\end{pmatrix}.\
$$
In all of these cases, we see that the exponential growth rates of the matrix product along the orbit
(and hence of the restriction of the Perron-Frobenius cocycle to $V$) are $\log\big((1+e^2)^pe^{1-p}\big)$
and $\log\big((1+e^2)^{1-p}e^p\big)$ as computed above. In the case where $p=\frac 12$, the two
exponents are both equal to $\log\big((1+e^2)^{\frac 12}e^{\frac 12}\big)$ as expected.

\section{Background}\label{sec:background}
In this section, we collect a number of theorems and definitions that we will need for the proof, as well
as setting out a number of related articles in the literature.

If $(Z,S)$ is a subshift, $\mathsf A(Z)$ is its \emph{alphabet} (so that $Z\subset \mathsf A(Z)^\Z$) and 
$\mathsf L(Z)$ denotes its \emph{language}, that is the set of all finite strings that appear in points of
$Z$. A point $z\in Z$ is said to be \emph{right transitive}
if $\{S^n(z):n\ge 0\}$ is dense in $Z$.

If $\pi\colon X\to Z$ is a factor map from a shift of finite type to a sofic shift and
$\nu$ is an ergodic invariant measure on $Z$, Petersen, Quas and Shin \cite{PQS}
established that the collection of ergodic invariant measures of relative maximal entropy is finite.
These measures are the relative equilibrium states
in the case where the potential function $\phi$ is taken to be 0. In the case
where the factor map $\pi$ is a one-block map (i.e. $\pi(x)_0$ depends only on $x_0$), 
they established that the 
number of ergodic measures of
relative maximal entropy over any ergodic invariant measure $\nu$ on $Z$
is bounded above by $\min_{j\in\mathsf A(Z)}|\rho^{-1}(j)|$, where 
$\rho$ is the symbol map giving rise to $\pi$. 
This result shows that,
in particular, the number of these measures is finite. The bound suffers from a failure to be invariant
under conjugacies. This deficiency was remedied and the bound improved in 
the paper \cite{MahsaAQ} of Allahbakshi and Quas, some ideas from which will play an
important role here. 

For $z\in Z$, if $x,x'\in \pi^{-1}z$, we say that $x$ \emph{transitions to} $x'$,
and write $x\to x'$, if for all $n$, there exists $\bar x\in\pi^{-1}z$
such that $\bar x_{-\infty}^n =x_{-\infty}^n$ and
$\bar x_m=x'_m$ for all sufficiently large $m$. We then define an equivalence relation on $\pi^{-1}z$ by
$x\leftrightarrow x'$ if $x\to x'$ and $x'\to x$. The equivalence classes are called \emph{transition classes} 
(a pigeonhole argument using the finite type property shows there are finitely many transition classes). 
Let $\mathcal T(z)$ denote the collection of transition classes over $z$. 
The paper \cite{MahsaAQ} establishes that the number of transition classes over any right 
transitive point $z\in Z$
is a constant $c_\pi$ independent of $z$. This constant is called the \emph{class degree} of $\pi$.

\begin{thm}[Allahbakhshi and Quas \cite{MahsaAQ}]
Let $X$ be a shift of finite type and $Z$ be a sofic shift. Let $\pi\colon X\to Z$ be a one-block
factor map. There exists a word $W=w_0^{n-1}$ in $\mathsf L(Z)$, a position $0\le l<n$, and a subset $B\subset\mathsf A(X)$ whose cardinality is the class
degree $c_\pi$, so that for each element $u_0^{n-1}$ of $\pi^{-1}(W)$, 
there is a word $v_0^{n-1}\in \pi^{-1}(W)$ such that $u_0=v_0$, $u_{n-1}=v_{n-1}$ and
$v_l\in B$. 

The number of measures of relative maximal entropy over $\nu$ is bounded above by $c_\pi$.
\end{thm}

The bound on the number of measures of relative maximal entropy
was extended by Allahbakhshi, Antonioli and Yoo \cite{AAY} to the number of relative equilibrium states
of a H\"older continuous (or Bowen) potential function.

In the situation described in the above theorem, $W$ is called a 
\emph{minimal transition block}; $B$ is a set of \emph{representatives} and the word $u$
is said to be \emph{routed through} $v_l$. (The minimality in the name refers to the fact that 
the set of representatives is as small as possible). A pair of elements $x,x'$ of $X$ is said to be
\emph{mutually separated} if $x_n\ne x'_n$ for each $n$. Two subsets $S_1$ and $S_2$ 
of $X$ are mutually separated if for each $x\in S_1$ and $x'\in S_2$, $x$ and $x'$ are 
mutually separated.

\begin{thm}[Allahbakhshi, Hong and Jung \cite{AHJ}]\label{thm:AHJ}
Let $\pi\colon X\to Z$ be a one-block factor map from an irreducible two-sided shift of finite type $X$ to 
a two-sided sofic shift $Z$. If $z\in Z$ is right transitive, then the elements of
$\mathcal T(z)$ are mutually separated. In particular, for each copy of $W$ in $z$,
there exists a bijection between $\mathcal T(z)$ and $B$ so that for 
each $C\in\mathcal T(z)$, there exists a representative $s\in B$ such that 
each $x\in C$ may be routed through $s$ over that copy of $W$ and through no other element of $B$.
\end{thm}

\begin{thm}[Yoo  \cite{Yoo-Factor}]\label{thm:Yoo-factor}
Let  $X$ be an irreducible two-sided shift of finite type, $Z$ a two-sided sofic shift, and
$\pi\colon X\to Z$ be a continuous factor map. Then 
there is a sofic shift $Y$ and factorization of $\pi\colon X\to Z$ as a composition of factor maps,
$\pi_2\circ \pi_1$ where $\pi_1\colon X\to Y$ and $\pi_2\colon  Y\to Z$ with the properties that
$\pi_2$ is finite-to-one of degree $c_\pi$, the class degree of $\pi$
and $\pi_1$ is of class degree 1. 
\end{thm}

\begin{thm}[Yoo \cite{Yoo-multiplicity}]\label{thm:Yoo-mult}
Let $\pi$ be a finite-to-one continuous factor map from a homeomorphism $S$ of a 
compact metric space $Y$ to a homeomorphism $T$ of a compact metric space $Z$.
Suppose that $\nu$ is an ergodic $T$-invariant measure.
Then:

\begin{itemize}
\item there exists $d\in\N$ such that for $\nu$-a.e.\ $z\in Z$, $|\pi^{-1}z|=d$;
\item there are only finitely many ergodic measures $\mu_1,\ldots,\mu_k$ on $Y$ 
such that $\pi_*\mu_i=\nu$ and $k\le d$;
if $\nu$ is fully supported, then so are the $\mu_1,\ldots,\mu_k$;
\item there exist multiplicities $m_1,\ldots,m_k$ whose sum is $d$.
\end{itemize}
In the case where $Y$ is a shift space and $\pi$ is a one-block map, there exists a joining
$\bar\mu$ on $Y^d$ such that for 
$\bar\mu$-a.e.\ $(y^1,\ldots,y^d)$, $\pi(y^1)=\ldots=\pi(y^d)$; the $y^i$ are mutually separated;
and $y^{M_i+1},\ldots,y^{M_i+m_i}$ are generic for $\mu_i$, where $M_i=m_1+\ldots+m_{i-1}$. 
\end{thm}

The joining $\bar\mu$ constructed in the above theorem is called an \emph{ergodic degree joining}.

The following theorem gives a criterion for simplicity of the top Lyapunov exponent of an operator
cocycle based on contraction of cones and Birkhoff's theorem on contraction of the Hilbert metric.
Recall that a \emph{cone} is a closed subset $\mathcal C$ of a real Banach space $B$ that is closed 
under addition and scalar multiplication by a non-negative real number.

For $f,g\in\mathcal C$, let $m(f,g)=\sup\{t\ge 0\colon f-tg\in \mathcal C\}$ and let
$M(f,g)=\inf\{s\ge 0\colon sg-f\in\mathcal C\}$. The projective distance between two 
points in the cone is defined as $\Theta_{\mathcal C}(f,g)=\log (M(f,g)/m(f,g))$. (Note that this is not
a metric as it may be infinite; also $\Theta_{\mathcal C}(\beta f,\gamma g)=
\Theta_{\mathcal C}(f,g)$ for all $\beta,\gamma>0$). The \emph{diameter}
of a subset $S$ of $\mathcal C$ is $\sup_{f,g\in S\setminus\{0\}}\Theta_{\mathcal C}(f,g)$.
A cone is said to be \emph{$D$-adapted} if whenever $f\in B$ and $g\in\mathcal C$,
then $g\pm f\in\mathcal C$ implies $\|f\|\le D\|g\|$.

\begin{thm}[Horan \cite{Horan}, Theorem 2.14]\label{thm:Joseph}
Let $Y$ be a compact metric space and $S\colon Y\to Y$ be a continuous 
invertible transformation. 
Let $\nu$ be an ergodic $S$-invariant Borel probability measure on $Y$. 
Let $B$ be a Banach space and let $\mathcal C$ be a $D$-adapted cone in $B$ 
such that $\mathcal C-\mathcal C=B$, $\mathcal C\cap (-\mathcal C)=\{0\}$.

Suppose that for each $y\in Y$, $\LL_y$ is a linear operator from 
$B$ to $B$ such that $y\mapsto\LL_y$
is continuous (where the linear operators on $B$ are equipped with the norm topology),
that $\LL_y(\mathcal C)\subset\mathcal C$ for each $y$
and that there
is a measurable subset $A\subset Y$ with $\nu(A)>0$ and an $n>0$ such that 
$\diam(\LL^{(n)}_y\mathcal C)<\infty$
for all $y\in A$. 

Then the leading Lyapunov exponent of the cocycle $(\LL^{(n)}_y)_{y\in Y}$ is simple.
That is there exist $\alpha>\beta$, a measurable function $v\colon Y\to B$
and a measurable function $\psi\colon Y\to B^*$ such that
$\LL_y(v(y))$ is a multiple of $v(S(y))$; $\frac 1n\log\|\LL_y^{(n)}v(y)\|\to\alpha$
a.e.; and $\limsup_{n\to\infty}\frac 1n\|\log\LL_y^{(n)}w\|\le\beta$ whenever $w\in\ker\psi(y)$.
\end{thm}

This theorem should be thought of as a skew product version of the Perron-Frobenius theorem.

We will use the relative variational principle of Ledrappier and Walters \cite{LedrappierWalters}.
Recall the Bowen definition of pressure:
$$
P(\phi)=
\lim_{\epsilon\to 0}\limsup_{n\to\infty} \frac 1n\log\sup_E\sum_{x\in E}e^{S_n\phi(x)},
$$
where the supremum is taken over $(n,\epsilon)$-separated sets, that is sets $E$ such that 
for any distinct elements $x,x'$  of $E$, there is $0\le j<n$ such that 
$d_\beta(T^jx,T^jx')\ge\epsilon$. In the case of shift spaces, this may be
simplified, fixing $\epsilon$ to be 1 and taking $E$ to be any set consisting of 
exactly one point in each cylinder set of length $n$ (so that $E$ has the same
cardinality as $\mathcal L_n(X)$). For symbolic systems,
$$
P(\phi)=\limsup_{n\to\infty}\frac 1n\log\sum_{x\in E}e^{S_n\phi(x)},
$$
where $E$ is any set with one representative of each cylinder set of length $n$.
This definition is further refined by restricting the elements of $E$ to lie in a fixed subset $K$:
$$
P(\phi,K)=\limsup_{n\to\infty}\frac 1n\log\sum_{\omega\in E\subset K}e^{S_n\phi(\omega)},
$$
where $E$ is any maximal $(n,1)$-separated collection of points of $K$.
We define $p_n(\phi,K)=\sup_{E\subset K;\ (n,1)\text{-separated}}
\sum_{\omega\in E}e^{S_n\phi(\omega)}$
so that 
$$
P(\phi,K)=\limsup_{n\to\infty}\frac 1n\log p_n(\phi,K).
$$ 

\begin{thm}[Relative Variational Principle]\label{thm:relVP}
Let $T\colon X\to X$ and $S\colon Y\to Y$ be continuous dynamical systems on compact spaces;
let $\nu$ be an ergodic invariant measure for $S$ and let $\pi\colon X\to Y$ be a continuous factor
map from $(X,T)$ to $(Y,S)$. 
Then for $\nu$-a.e.\ $y$, $P(\phi,\pi^{-1}y)=r_{\pi}(\nu)$.
\end{thm}

\section{Proofs}
In this section, we start with some preliminary lemmas and
then establish Theorem \ref{thm:cd1case} (which is 
the special case of the main theorem in the case where $\pi$ has class degree 1), before using it to
prove the main theorem.

The proof structure is as follows. We start with a factor map $\pi\colon X\to Z$ and an ergodic invariant measure
$\nu$ on $Z$. Given a $\nu$-typical point $z\in Z$,
its preimages in $X$ can be separated into a number of transition classes as described in the previous section. 
Those results show that one can associate pressure-maximizing measures on $X$ to these classes, 
and that $\nu$-a.e.\ $z$ gives rise to the same collection of measures on $X$. 
Theorem \ref{thm:cd1case} deals with the case where the class degree is 1 
(so there is a single transition class). Some preparatory lemmas show that the 
the cocycle of operators maps a family of cones inside itself, and from time to time maps a cone in the family into a finite diameter
sub-cone of the cone in the family. This allows us to apply Theorem \ref{thm:Joseph} showing that there is a simple leading Lyapunov
exponent. A calculation shows that this exponent is the quantity appearing on the left side of the equality in 
the Relative Variational Principle (while the conclusion of Theorem \ref{thm:cd1case}
is that the exponent is the right side of
the equality). To deal with the case of class degree greater than 1, we express $\pi$ as $\pi_2\circ\pi_1$ 
as in Theorem \ref{thm:Yoo-factor}, and express the Perron-Frobenius cocycle as a sum of non-interacting cocycles,
each of which satisfies the hypotheses of Theorem \ref{thm:cd1case}, with one summand per transition class. 

Finally, in the case where $\phi$ is locally constant, there is a corresponding family of locally constant functions
that is mapped into itself by the Perron-Frobenius cocycle. We show that this family intersects each of the cones 
described above, so that the multiplicity of the top Lyapunov exponent is captured by the action on this finite-dimensional subspace. 

For this section, let $\phi$ be a fixed H\"older continuous function. 
Given $\beta<1$, we define a semi-norm  on $C^\beta(X^+)$ by
$|f|_\beta=\sup_{x\ne x'}|f(x)-f(x')|/d_\beta(x,x')$ (that is the Lipschitz constant of $f$ with respect
to $d_\beta$) and a norm by
$\|f\|_\beta=\max(\|f\|_\infty,|f|_\beta)$. 
Let $\beta$ be such that $\|\phi\|_\beta<\infty$. This quantity will be fixed from here on.
We also assume throughout this section that the factor map $\pi$ is a one-block map as this is the 
context in the proof of the main theorem.

We define a family of cones, one for each real $a>0$, by 
$$
\mathcal C_a=
\left\{f\in C^\beta(X^+)\colon f\ge 0; f(x')\le e^{ad_\beta(x,x')}f(x)\text{ whenever $x_0=x'_0$}
\right\}.
$$
These cones are widely used in symbolic dynamics and appear, for instance, in the work of Parry and 
Pollicott \cite{ParryPollicott}, although our usage differs slightly as we do not impose any condition
on $f(x)/f(x')$ when $x_0\ne x_0'$. This is important for us, since some operators that we consider
yield functions that are 0 on part of $X^+$.

\begin{lem}\label{lem:conetocone}
Let $a$ be large enough that
$b:=\beta(a+|\phi|_\beta)<a$.
Then $\LL_j\mathcal C_a\subset \mathcal C_b$ for each 
$j\in\mathsf A(Y)$.
\end{lem}

\begin{proof}
Let $f\in \mathcal C_a$. For each symbol $i\in \mathsf A(X)$, set $\tilde\LL_if(x)=e^{\phi(ix)}f(ix)$. 
Suppose $x$ and $x'$ agree for $n$ symbols for some $n\ge 1$ and suppose $f(ix')>0$ (so that
$f(ix)>0$ also). Then 
\begin{align*}
\frac{\tilde\LL_i f(x')}{\tilde\LL_if(x)}&=\frac{e^{\phi(ix')}f(ix')}{e^{\phi(ix)}f(ix)}\\
&\le e^{|\phi|_\beta\beta^{n+1}}e^{a\beta^{n+1}}\\
&\le e^{\beta(|\phi|_\beta+a)d_\beta(x,x')}=e^{bd_\beta(x,x')}.
\end{align*}
Since for $j\in \mathsf A(Y)$, $\LL_j=\sum_{i\in\pi^{-1}j}\tilde\LL_i$ (where $\pi^{-1}j$ denotes
the symbols in $\mathsf A(X)$ that map to $j$ under the alphabet map defining $\pi$), the result follows. 
\end{proof}

\begin{lem}\label{lem:normcomp}
For $f\in \mathcal C_a$, $\|f\|_\beta \le \max(3,1+ae^a)\|f\|_\infty$.
It follows that $\mathcal C_a$ is $D$-adapted with $D=\max(6,2+2ae^a)$.
\end{lem}

\begin{proof}
Let $f\in C_a$. If $x,x'\in X^+$ have different initial symbols, then $|f(x)-f(x')|
\le 2\|f\|_\infty\le \max(2,ae^a)\|f\|_\infty d_\beta(x,x')$. 
If they have the same initial symbol, then $|f(x)-f(x')|\le |f(x)|(e^{ad_\beta(x,x')}-1)
\le \|f\|_\infty ae^ad_\beta(x,x')\le \max(2,ae^a)\|f\|_\infty d_\beta(x,x')$, 
where we used the mean value theorem for
the second inequality. Hence $|f|_\beta\le \max(2,ae^a)\|f\|_\infty$ so that
$\|f\|_\beta\le \max(3,1+ae^a)\|f\|_\infty$.

For the second statement in the lemma, $g\pm f\in\mathcal C_a$ implies $\|f\|_\infty\le \|g\|_\infty$, so that
$\|g\pm f\|_\infty\le 2\|g\|_\infty$ and $\|g\pm f\|_\beta\le \max(6,2+2ae^a)\|g\|_\infty
\le \max(6,2+2ae^a)\|g\|_\beta$. Subtracting $g-f$ from $g+f$, we obtain the desired bound. 
\end{proof}

For these cones, we have the following lemma (which can be seen as a special case of a result 
of And\^o \cite{Ando}). Expressing arbitrary H\"older continuous functions 
as a difference of elements of the cone will allow us to prove simplicity of the top Lyapunov 
exponent.
\begin{lem}\label{lem:Ando}
For all $f\in C^\beta(X^+)$, there exist $g,h\in \mathcal C_a$ with 
$\|g\|_\beta,\|h\|_\beta\le (2+\frac 1a)\|f\|_\beta$ such that $f=g-h$.
\end{lem}

\begin{proof}
Let $f\in C^\beta(X^+)$, 
let $g=f+(1+\frac 1a)\|f\|_\beta$ and $h=(1+\frac 1a)\|f\|_\beta$. 
Clearly $h\in \mathcal C_a$. Notice that $\min g\ge \frac 1a\|f\|_\beta$, so that 
\begin{align*}\frac{g(x)}{g(x')}&=1+\frac{g(x)-g(x')}{g(x')}\le 1+\|f\|_\beta d_\beta(x,x')/(\|f\|_\beta/a)\\
&= 1+ad_\beta(x,x')\le e^{ad_\beta(x,x')}.
\end{align*}
In particular, $g\in\mathcal C_a$ and $\|g\|_\beta$, $\|h\|_\beta$ are bounded above by
$(2+\frac 1a)\|f\|_\beta$. 
\end{proof}

\begin{lem}\label{lem:finitediam}
Let $0<b<a$ and $A\ge 1$. Let $P$ be a subset of the alphabet of $X$. Write
$[P]=\bigcup_{j\in P}[j]$ and 
$$
S_P=\left\{f\colon f(x)>0\text{ iff $x\in[P]$;  } f(x)\le Af(x')\text{ for all $x,x'\in [P]$}\right\}.
$$
Then there exists $K>0$ such that $\Theta_{\mathcal C_a}(f,g)\le K$ for all $f,g\in S_P\cap \mathcal C_b$. 
\end{lem}

The conclusion here states that the diameter of the set is finite. This is a key hypothesis in Birkhoff's cone
contraction argument. 

\begin{proof}
Let $t>0$ be chosen sufficiently small to ensure that $\frac{2btA}{1-tA}\le a-b$. 
Let $f,g\in S_P\cap \mathcal C_b$. Using the scale-homogeneity of $\Theta_{\mathcal C_a}$, 
we may scale $f$ and $g$ so that $\min_{[P]} f=\min_{[P]} g=1$,
and hence $\max f,\max g\le A$. 

We claim that $f-tg\in\mathcal C_a$. 
Let $x,x'\in X^+$ have a common first symbol belonging to $P$ (if $x,x'$ have a common first symbol 
outside $P$, then $(f-tg)(x)$ is trivially bounded above by $e^{ad_\beta(x,x')}(f-tg)(x')$ since both of 
these quantities are zero). We have
\begin{align*}
\frac{f(x')-tg(x')}{f(x)-tg(x)}&\le \frac{e^{bd_\beta(x,x')}f(x)-te^{-bd_\beta(x,x')}g(x)}{f(x)-tg(x)}\\
&= e^{bd_\beta(x,x')}+\frac{t(e^{bd_\beta(x,x')}-e^{-bd_\beta(x,x')})g(x)}{f(x)-tg(x)}\\
&\le e^{bd_\beta(x,x')}\left(1+\frac{At}{1-At}(1-e^{-2bd_\beta(x,x')})\right)\\
&\le e^{bd_\beta(x,x')}\left(1+\frac{2Atb}{1-At}d_\beta(x,x')\right)\\
&\le e^{bd_\beta(x,x')}e^{(a-b)d_\beta(x,x')}=e^{ad_\beta(x,x')},
\end{align*}
so that $f-tg\in \mathcal C_a$. By symmetry, $g-tf\in\mathcal C_a$, or equivalently $(1/t)g-f\in \mathcal C_a$. 
Hence $\Theta_{\mathcal C_a}(f,g)\le \log(1/t^2)$ for all $f,g\in S_P\cap \mathcal C_b$.
%
%
\end{proof}

\begin{lem}\label{lem:positiveset}
Let $\pi\colon X\to Y$ be a factor map of class degree 1.
Let $W=w_0^{n-1}$ be a minimal transition block in $Y$. Then there exists an $A\ge 1$ such that
for any $y\in[W]$, and any $f\in \mathcal C_a$, 
$\mathcal L^{(n)}_y f\in S_P\cup\{0\}$, where $S_P$ is the set in 
Lemma \ref{lem:finitediam} and $P$ is $\{j\in \mathsf A(X)\colon \exists U\in\pi_b^{-1}(W)
\colon Uj\in \mathsf L(X)\}$.
\end{lem}

\begin{proof}
Since $S_P\cup\{0\}$ is closed under addition and $\mathcal L_y^{(n)}$ is linear,
it suffices to prove the statement for a function $f$
supported on a single cylinder set. Suppose that $f$ is supported on $[k]$. If there is no preimage of $W$
whose initial symbol is $k$, we see that $\mathcal L_y^{(n)}f=0$ since there are no positive summands. 
Suppose on the other hand that $U$ is a preimage of $W$ under $\pi$ starting with a $k$. 
Let $j\in P$ and let $V$
be a preimage of $W$ under $\pi$ such that $Vj\in \mathsf L(X)$. 
Since $W$ is a minimal transition block with a single representative,
there exists a preimage $U'$ of $W$ starting with the first symbol of $U$ 
and ending with the last symbol of $V$. 
If $x\in [j]$, We now calculate
$$
\mathcal L^{(n)}_yf(x)\ge e^{S_n\phi(U'x)}f(U'x)\ge e^{n\min\phi}\|f\|_\infty/e^a.
$$
On the other hand, it is clear that $\mathcal L^{(n)}_yf(x)\le |\mathsf A(X)|^ne^{n\max\phi}\|f\|_\infty$
for any $x\in X$. Hence we have demonstrated the hypothesis of Lemma \ref{lem:finitediam} is satisfied with 
$A=e^{a+n(\max\phi-\min\phi)}|\mathsf A(X)|^n$.
\end{proof}

We point out that the idea of studying the Ruelle-Perron-Frobenius cocycle over a factor $Y$
and expressing the operators corresponding to symbols in $\mathsf A(Y)$ 
as sums of operators indexed by symbols in $\mathsf A(X)$, as well as some of the cones
that we study here and the description of $S_P$ above, appears in work of Piraino \cite{Mark}.

\begin{thm}\label{thm:cd1case}[Main theorem, class degree 1 case]
Let $X$ be an irreducible shift of finite type, let $\pi\colon X\to Y$ be a forward-looking 
factor map of class degree 1 and 
let $\phi$ be a H\"older continuous function on $X^+$.
Suppose $\nu$ is a fully supported invariant measure on 
$Y$. 
Then the cocycle $(\mathcal L_{y}^{(n)})$ 
has a simple top Lyapunov exponent, whose value is $r_{\pi}(\nu)$, the relative pressure of
$\phi$ over $\nu$. 

Further, for $\nu$-a.e.\ $y$, 
$$
\lim_{n\to\infty}\tfrac 1n\log\|\LL_y^{(n)}\mathbf 1\|=
\lim_{n\to\infty}\tfrac 1n\log\|\LL_y^{(n)}\mathbf 1_{\pi^{-1}[y_0]}\|
=r_\pi(\nu).
$$ 
\end{thm}

\begin{proof}
By conjugating $X$ and $Y$ if necessary, we may assume that $\pi$ is a one-block map.
This does not affect any of the hypotheses or conclusions of the theorem (see 
\cite[Proposition 1.5.12]{LindMarcus},\cite{MahsaAQ} for more details).
Let $\beta$ be such that $\|\phi\|_\beta<\infty$ and let
$a$ satisfy the hypothesis in Lemma \ref{lem:conetocone}.
Let $W$ be a minimal transition block for the factor map $\pi\colon X\to Y$. 
By Lemma \ref{lem:Ando}, $\mathcal C_a-\mathcal C_a=C^\beta(X^+)$. 
By Lemmas \ref{lem:conetocone}, \ref{lem:finitediam} and \ref{lem:positiveset},
we see that $\mathcal L_y^{(|W|)}\mathcal C_a$ is a finite diameter subset 
of $\mathcal C_a$ for any $y\in[W]$. 
Since the hypotheses of Theorem \ref{thm:Joseph} are satisfied (the continuity of $y\mapsto\LL_y$
is because the map is piecewise constant and the $D$-adaptedness
condition on $\mathcal C_a$ is satisfied by the
second statement of Lemma \ref{lem:normcomp}), the top Lyapunov exponent 
of the cocycle $(\LL_y^{(n)})$
acting on $C^\beta(X^+)$ is simple.

Notice that for $y\in Y$, and $g\in\mathcal C_a$,
$$
\LL_y^{(n)}g(x)=
\sum_{W\in\pi^{-1}(y_0^{n-1})\colon Wx_0\in\mathsf L(X)}e^{S_n\phi(Wx)}g(Wx),
$$
so that $\|\LL_y^{(n)}g\|_\infty$ is bounded above by $p_n(\phi,\pi^{-1}y)\|g\|_\infty$.
By Lemmas \ref{lem:conetocone} and  \ref{lem:normcomp}, 
$\|\LL_y^{(n)}g\|_\beta$ is bounded above by 
$\max(3,1+e^a)p_n(\phi,\pi^{-1}y)\|g\|_\infty$.
If $f\in C^\beta(X^+)$, using Lemma \ref{lem:Ando}, we may write $f$ as the difference $g-h$ 
with $g,h\in \mathcal C_a$, each of $\|\cdot\|_\beta$ norm at most $(2+\frac 1a)\|f\|_\beta$.
Hence $\|\LL_y^{(n)}f\|_\beta\le 2(2+\frac 1a)\max(3,1+e^a)p_n(\phi,\pi^{-1}y)\|f\|_\beta$.
As noted above, we have
$\limsup_{n\to\infty}\frac 1n\log p_n(\phi,\pi^{-1}y)=P(\phi,\pi^{-1}y)$, which is
$r_\pi(\nu)$ by Theorem \ref{thm:relVP},
so that the top Lyapunov exponent is bounded above by $r_\pi(\nu)$. 

For the converse inequality, let $(x^i)_{i\in\mathsf A(X)}$ be a collection of points in $X^+$, 
where $x^i$ starts with the symbol $i$. Now 
$$
\sum_{i\in\mathsf A(X)}\LL^{(n)}_y\mathbf 1(x^i)\ge e^{-c}p_n(\phi,\pi^{-1}y),
$$
where $c$ is a constant independent of $n$ where $|S_n\phi(x)-S_n\phi(x')|\le c$ whenever 
$x_0^{n-1}={x'}\,_0^{n-1}$ (such a $c$ exists since $\phi$ is H\"older). Then
$$
\left\|\LL^{(n)}_y\mathbf 1\right\|_\beta\ge
\left\|\LL^{(n)}_y\mathbf 1\right\|_\infty\ge \frac{e^{-c}}{|\mathsf A(X)|}p_n(\phi,1,\pi^{-1}y).
$$
In particular, for $\nu$-a.e.\ $y$, the limit superior growth rate of $\|\LL^{(n)}_y\mathbf 1\|_\beta$
is at least $r_\pi(\nu)$, as required.
\end{proof}

\begin{proof}[Proof of Theorem \ref{thm:main}]
We assume without loss of generality as above that $\pi$ is a one-block map. 
Using Theorem \ref{thm:Yoo-factor}, $\pi:X\to Z$ may be factorized as $\pi_2\circ\pi_1$, where
$\pi_1$ is of class degree 1 from $X$ to a sofic shift $Y$; and $\pi_2$
is finite-to-one, and for $\nu$-a.e.\ point, $\pi_2^{-1}(z)$ consists of $c_\pi$ pre-images.

We need a more precise description of the construction of the
sofic shift $Y$ and the factor code $\pi_1$, for which we will follow \cite{Yoo-Factor}. The space
$Y$ is built from a minimal transition block $W$ in $\mathsf L(Z)$. Recall the 
representatives of the transition block are a subset $B$ of $\mathsf A(X)$ of cardinality 
$c_\pi$ such that if $\pi(x)\in [W]$, then $x$ may be locally modified on the coordinates $(0,n-1)$
to give a point $x'\in X$ with $x'_l\in B$. 

The alphabet of $Y$ is then $\mathsf A(Z)\times (B\cup\{\star\})$. The factor map $\pi_1$ is defined
as follows:
$$
\pi_1(x)_m=\begin{cases}
(\pi(x)_m,s)&\text{$\pi(x)_{m-l}^{m-l+n-1}=W$, $x_{m-l}^{m-l+n-1}$ routable through $s$;}\\
(\pi(x)_m,\star)&\text{$\pi(x)_{m-l}^{m-l+n-1}\ne W$.}
\end{cases}
$$
That is, $\pi_1(x)$ records the image in $Z$, together with the representatives in $B$ through which 
the orbit of $x$ may be routed each time that the orbit passes through a transition block.
The factor map $\pi_2$ is the one-block factor map from $Y$ to $Z$ defined by the 
symbol map sending $(a,b)$ to $a$ for any $(a,b)\in\mathsf A(Z)\times  (B\cup\{\star\})$. 

We then define an operator cocycle over $Y$. For each $s\in B$, let $R_s\subset \mathsf A(X)$
be the collection of symbols in $X$ that a preimage of $W$ may pass through if it is routable through $s$. 
By Theorem \ref{thm:AHJ}, these sets are disjoint. Write $[R_s]$ for $\bigcup_{i\in R_s}[i]$ and let 
$q$ be the $l$th symbol of $W$.

The generator of the cocycle is then defined by
\begin{align*}
\bar\LL_{(j,\star)}f(x)&=\LL_jf(x)\text{ for $j\in\mathsf A(Z)$;}\\
\bar\LL_{(q,s)}f(x)&=\LL_q(\mathbf 1_{[R_s]}f)(x)\text{ for $s\in B$.}
\end{align*}
That is, each time $\pi(x)$ passes through a transition block, the operator projects to 
the part of the function routable through the specified representative. Note that 
$(q,\star)\in\mathsf A(Y)$, and this appears in the image of points under $\pi_1$ for
points in $X$ whose symbol maps to $q$, but where the word $\pi(x)_{m-l}^{m-l+n-1}$
is not equal to $W$.

By Theorem \ref{thm:Yoo-mult}, there are finitely many ergodic invariant measures on $Y$
projecting to $\nu$, say $\mu_1,\ldots,\mu_k$, each fully supported; as well as multiplicities $m_1,\ldots,m_k$ 
summing to $c_\pi$ such that a $\nu$-generic $z\in Z$ has $m_i$ $\mu_i$-generic pre-images
under $\pi_2$ for each $i$, with the whole collection of $c_\pi$ pre-images mutually separated.
Further, there exists an ergodic measure $\bar\mu$ on $Y^{c_\pi}$
where $\bar\mu$-almost every point is supported on the $c_\pi$ pre-images of some point $z\in Z$;
the first $m_1$ being generic points for the ergodic measure $\mu_1$ on $Y$; the next $m_2$ being generic
for the measure $\mu_2$ etc. We assume without loss of generality that $r_{\pi_1}(\mu_1)\ge
r_{\pi_1}(\mu_2)\ge\ldots$, where $r_{\pi_1}(\mu_i)$ is the $\pi_1$-relative pressure of $\phi$ 
over $\mu_i$; and that the maximal value of $r_{\pi_1}(\mu_i)$ is attained for
$i=1,\ldots,p$ (but not for $i=p+1,\ldots,k$). Notice that $r_{\pi_1}(\mu_i)=r_\pi(\nu)$
for $i=1,\ldots,p$ since any ergodic measure on $X$ in $\pi_*^{-1}\{\nu\}$ lies in one of the 
${\pi_1}_*^{-1}\{\mu_i\}$ for some $\mu_i$; and $\pi_2$ is finite-to-one, so does not decrease entropy.

Since $\pi_1\colon X\to Y$ and each $\mu_i$ satisfies the conditions of Theorem \ref{thm:cd1case},
we see that there is a simple top exponent $\lambda_i$ for the cocycle
$(\bar\LL_y^{(n)})_{y\in Y}$ for each of the measures $\mu_i$. The set of $y\in Y$ for which 
the exponent $\lambda_i$ is achieved and for which the second Lyapunov exponent is
strictly smaller is a collection of full $\mu_i$-measure.

For $\bar\mu$-a.e.\ $(y^1,\ldots,y^{c_\pi})\in Y^{c_\pi}$, the simple top exponent of 
the cocycle $(\bar\LL_{y})^{(n)}$ is $\lambda_i$ for each $y=y^{M_i+k}$ with $k=1,\ldots,m_i$ 
(where $M_i=m_1+\ldots+m_{i-1}$ and $M_1=0$). 
In particular, the top exponent of the cocycle is almost surely simple with exponent
$\lambda_1=r_{\pi_1}(\mu_1)$ over each of $y^1,\ldots,y^{m_1+\ldots+m_p}$
and strictly smaller for the other $y$'s. 

We now derive a relationship between the cocycle $(\mathcal L^{(n)}_z)$ over $Z$ 
and the cocycle $(\bar{\mathcal L}^{(n)}_y)$ over $Y$.
Recall that for $\bar\mu$-a.e.\ $\bar y=(y^1,\ldots,y^{c_\pi})$, one has the 
equality $\pi(y^1)=\ldots=\pi(y^{c_\pi})$. 
Write $\bar\pi(\bar y)$ for this common value.
Next, we claim that for $\bar\mu$-a.e.\ $\bar y$, 
\begin{equation}\label{eq:Lsum}
\LL_{\bar\pi(\bar y)}^{(n)}f(x)=
\sum_{i=1}^{c_\pi}\bar\LL^{(n)}_{y^i}f(x)
\end{equation}
for all $n$ such that $\bar\pi(\bar y)_0^{n-1}$ contains a copy of $W$, the minimal transition block
used in the definition of $\pi_1$. 

To see this, notice that if $q=w_l$ is the symbol in $W$ over which the representatives lie,
then the following identities hold
\begin{align*}
\LL_q&=\sum_{i\in S}\bar\LL_{q,i}\\
\LL_j&=\bar\LL_{j,\star}\text{ for each $j\in\mathsf A(Z)$ (including $q$)}.
\end{align*}
So the composition of the $\LL_{z_i}$ is a composition 
in which a number of the terms (those occurring when $z$ contains a copy of $W$) are replaced
by a sum of $\bar\LL_{q,i}$. Since the $\bar\LL$ are linear, we may distribute the 
composition over the sum.
Since when $z$ is right transitive, the transition classes are mutually separated
(Theorem \ref{thm:AHJ}), almost all of the terms in the summation vanish; the only ones that survive
are those in which the choices of representative are consistent: the representative over one instance of $W$
together with the point $z$ determines the representative over all of the other instances of $W$ by virtue of
the mutual separation of the classes in $\pi^{-1}(z)$.

For $\bar\mu$-a.e.\ $\bar y=(y^1,\ldots,y^{c_\pi})$, each of the $y^i$'s is right transitive; and
the map $\pi_1\colon X\to Y$ is of class degree 1. The pre-images $\pi_1^{-1}(y^i)$ for $i=1,\ldots,c_\pi$
form the transition classes, $\mathcal T(\bar\pi(\bar y))$ in $X$ over $\bar\pi(\bar y)\in Z$. 
By Theorem \ref{thm:cd1case}, applied to $\pi_2\colon X\to Y$, for $\nu$-a.e.\ $\bar\pi(\bar y)$ and
each $y^i$ with $1\le i\le m_1+\ldots+m_p$,
the cocycle $(\bar\LL_{y^i}^{(n)})$ has an equivariant one-dimensional space of functions
growing at rate $r_{\pi_1}(\mu_1)$. In particular, the functions $\mathbf 1_{\pi_1^{-1}[y^i_0]}$
for $i=1,\ldots,M_p$ grow at rate $r_{\pi_1}(\mu_1)$ under the respective
cocycles $\bar\LL_{y^i}^{(n)}$ (and are eventually annihilated by the other cocycles).
Since the $\mathcal T(\bar\pi(\bar y))$ are mutually separated, for each $n$, the
$\bar\LL_{y^i}^{(n)}\mathbf 1_{\pi_1^{-1}[y^i_0]}$ are disjointly supported. 
By \eqref{eq:Lsum}, the span of $\{\mathbf 1_{\pi_1^{-1}[y^i_0]}\colon 1\le i\le M_p\}$
is an $M_p$-dimensional space of functions, where the entire space grows under the cocycle
$\LL_{\bar y}^{(n)}$ at rate $r_{\pi_1}(\mu_1)=
r_{\pi}(\nu)$. 

Theorem \ref{thm:cd1case} implies that in any two-dimensional space of functions
supported on $\pi_1^{-1}(y^i)$, there is a function whose growth rate is strictly smaller
than $r_{\pi}(\nu)$. 

On the other hand, by Theorem \ref{thm:cd1case} for $i>M_p$, the growth rate on
$\pi_1^{-1}(y^i)$ is at most $r_{\pi_1}(\mu_{p+1})$, which is strictly smaller. Combining these facts,
it follows that the dimension of the fastest growing space is precisely $m_1+\ldots+m_p$ as required. 
\end{proof}

\begin{proof}[Proof of Corollary \ref{cor:locconst}]
For this proof, we assume that $\pi$ is a one-block factor map, given by the 
map $\rho\colon\mathsf{A}(X)\to\mathsf{A}(Z)$, (as in the previous theorem)
and the locally constant function $\phi(x)$ depends only on $x_0$. The key observation in this case is
that $\LL_z$ maps $\mathsf{LC}$,
the finite-dimensional subspace of functions depending only on the 0th coordinate into itself.

Specifically, if $f$ is the function taking the value $f_i$ on the cylinder $[i]$, then 
$$
\LL_z f(x)=\sum_{i:ix_0\in\mathsf{L}(X), \rho(i)=z_0} e^{\phi_i}f_i,
$$
another function whose value is determined by $x_0$. 
That is $\LL_z$ is represented by the matrix with entries
$$
(A_z)_{ij}=\mathbf 1_{ij\in\mathsf{L}(X)}\mathbf 1_{\rho(i)=z_0} e^{\phi_i}.
$$
Let $\pi=\pi_1\circ\pi_2$ as in the proof of Theorem \ref{thm:main}, so that 
the symbol map $\rho$ is the composition of maps $\rho_1$ and $\rho_2$. 
We use the notation of the proof of Theorem \ref{thm:main}.
Let $\bar y=(y^1,\ldots,y^{c_\pi})$ be an generic element of the degree joining,
where we assume that $y^1,\ldots,y^{M_p}$ are generic for measures
$\mu_1,\ldots,\mu_p$ with $r_{\pi_1}(\mu_i)=r_\pi(\nu)$ for $i=1,\ldots,p$. 
Then we showed above that $\bar\LL_{y_i}^{(n)}\mathbf 1_{\pi_1^{-1}[y^i_0]}$ grows 
at rate $r_\pi(\nu)$ for $i=1,\ldots,M_p$ and for each $n$, these functions are disjointly 
supported. Further, $\bar\LL_{y_i}^{(n)}\mathbf 1_{\pi_1^{-1}[y^i_0]}
=\LL_{\bar\pi(\bar y)}^{(n)}\mathbf 1_{\pi_1^{-1}[y^i_0]}$.

Since $\mathbf 1_{\pi_1^{-1}[y^i_0]}\in\mathsf{LC}$, we see that
the multiplicity of the exponent $r_\pi(\nu)$ for the matrix cocycle $A_z^{(n)}$ is
at least $M_p$. However, multiplicity of the exponent $r_\pi(\nu)$ 
for the action of the cocycle on $C^\alpha(X^+)$ is
an upper bound for the multiplicity on the subspace $\mathsf{LC}$. Hence
the multiplicity of the leading exponent for the matrix cocycle is exactly $M_p$ as required.
\end{proof}

\section{Acknowledgment}
We would like to acknowledge a very helpful referee's report that led to a substantial
improvement in the presentation.

\bibliographystyle{abbrv}
\bibliography{METMREbib}

\end{document}